\documentclass[oneside,11pt]{amsart}
\usepackage[pagebackref,breaklinks,unicode]{hyperref}
\hypersetup{pdftitle = Fine structures}
\title[Stable cylinders and fine structures]{Stable cylinders and fine structures for hyperbolic groups and curve graphs}

\usepackage[margin=3cm]{geometry}

\author{Harry Petyt}
\address{Mathematical Institute, University of Oxford, UK}
\email{petyt@maths.ox.ac.uk}

\author{Davide Spriano}
 \address{Mathematical Institute, University of Oxford, UK}
\email{spriano@maths.ox.ac.uk}

\author{Abdul Zalloum}
\address{Institute for Advanced Study in Mathematics, Harbin Institute of Technology, China}
\email{abdul.zalloum@utoronto.ca}

\usepackage{amsmath,amssymb,amsthm, mathrsfs, mathtools}
\usepackage{enumitem} \setlist{nosep} 
\usepackage[utf8]{inputenc}
\usepackage[english]{babel}
\usepackage[pagebackref,breaklinks,unicode]{hyperref}
\usepackage{float}
\usepackage{comment}
\usepackage{csquotes}
\usepackage[font=small,justification=centering]{caption}
\usepackage{subcaption}
\usepackage[dvipsnames]{xcolor}
\usepackage{mathabx} 
\usepackage{tikz-cd}

\usepackage{xfrac} 
\usepackage[normalem]{ulem} 
\usepackage{dirtytalk}
\usepackage{cancel}
\usepackage{marginnote}
\usepackage{endnotes}
\usepackage{xnotes}

\usepackage[T1]{fontenc} 
\usepackage[utf8]{inputenc} 

\newtheorem{theorem}{Theorem}[section]
\newtheorem{proposition}[theorem]{Proposition}
\newtheorem{corollary}[theorem]{Corollary}
\newtheorem{lemma}[theorem]{Lemma}

\theoremstyle{definition}
\newtheorem{definition}[theorem]{Definition}

\newtheorem{remark}[theorem]{Remark}

\newcounter{shcount}
\newcommand*{\bsh}[1]{\theoremstyle{definition}\newtheorem{subhead\theshcount}[theorem]{#1}
    \begin{subhead\theshcount}} 
\newcommand*{\esh}{\end{subhead\theshcount}\stepcounter{shcount}} 
\newcommand*{\ubsh}[1]{\theoremstyle{definition}\newtheorem*{subhead\theshcount}{#1}
    \begin{subhead\theshcount}} 
\newcommand*{\uesh}{\end{subhead\theshcount}\stepcounter{shcount}} 

\newcounter{claimcount}

\newenvironment{claim*}[1]{\par\vspace{2mm}\noindent
    \underline{Claim:}\hspace{2mm}#1}{}

\newcounter{enumlabelcount}
\makeatletter\newcommand\enumlabel[1][]{\item[#1]
    \refstepcounter{enumlabelcount}\def\@currentlabel{#1}}\makeatother

\renewcommand*{\backrefalt}[4]{\ifcase #1 (Not cited).\or (Cited p.~#2).\else (Cited pp.~#2).\fi} 

\makeatletter\def\subsection{\@startsection{subsection}{1}\z@{.7\linespacing\@plus\linespacing}
    {.5\linespacing}{\normalfont\scshape\centering}}\makeatother 
\makeatletter\def\part{\@startsection{part}{1}\z@{.7\linespacing\@plus\linespacing}
    {.5\linespacing}{\normalfont\large\scshape\centering}}\makeatother 

\DeclareMathOperator{\dist}{\mathsf{d}}

\DeclareMathOperator{\diam}{diam}
\DeclareMathOperator{\hull}{Hull}
\DeclareMathOperator{\isom}{Isom}
\DeclareMathOperator{\MCG}{MCG}

\newcommand*{\D}{\mathcal{D}}
\newcommand*{\E}{\mathcal{E}}

\newcommand*{\C}{\mathcal{C}}
\newcommand*{\g}{\mathfrak{g}}
\renewcommand*{\H}{\mathcal{H}}

\newcommand*{\N}{\mathcal N}

\newcommand*{\p}{\mathfrak{p}}

\newcommand*{\s}{\mathfrak S}

\newcommand{\eps}{\varepsilon}

\newcommand*{\cal}{\mathcal}

\newcommand*{\sgen}[1]{\langle#1\rangle}

\newcommand*{\ssm}{\smallsetminus}

\definecolor{harrycomment}{rgb}{0.6,0,0.4}

\usepackage{xcolor} 
\usepackage{endnotes} 

\definecolor{harrycolor}{rgb}{0.5, 0.0, 0.5} 
\definecolor{abdulcolor}{rgb}{1.0, 0.0, 0.0} 
\definecolor{davidecolor}{rgb}{0.0, 0.0, 1.0}  

\begin{document}

\begin{abstract}
In 1995, Rips and Sela asked if torsionfree hyperbolic groups admit globally stable cylinders. We establish this property for all residually finite hyperbolic groups and curve graphs of finite-type surfaces. These cylinders are fine objects, and the core of our approach is to upgrade the hyperbolic space to one with improved fine properties via a generalisation of Sageev's construction. The methods also let us prove that curve graphs of surfaces admit equivariant quasiisometric embeddings in finite products of quasitrees.
\end{abstract}

\maketitle

{\hypersetup{hidelinks}\setcounter{tocdepth}{1}\tableofcontents\setcounter{tocdepth}{2}}

\section{Introduction}


In their important work on equations over torsionfree hyperbolic groups, Rips--Sela introduced the notion of \emph{stable cylinders} in hyperbolic groups \cite{ripssela:canonical}. A cylinder in a hyperbolic group $G$ is simply a subset of a neighbourhood of a geodesic in $G$. Roughly speaking, a collection of cylinders is stable for a subset $A\subset G$ if for every triple in $A$, the cylinders between each pair are equal until they diverge in a tripodal fashion, except for in some controlled number of small balls (see Figure~\ref{fig:cylinders}).

Given a cylinder from 1 to $g\in G$, Rips--Sela construct what they call a ``canonical representative'' for $g$, which is a quasigeodesic from 1 to $g$. (It should be noted that the term is a misnomer, as the representatives are in multiple ways not canonical, although they are ``coarsely canonical'': see \cite[Def.~3.8, 3.9]{ripssela:canonical}.) Given cylinders that are stable for $A$, the triangles formed by representatives between elements of $A$ are tripods, except for in some controlled number of small balls \cite[Thm~3.11]{ripssela:canonical}.

Let $G$ be a torsionfree hyperbolic group. The main technical result proved in \cite[Thm~4.2]{ripssela:canonical} states that for each finite subset $A\subset G$, there exists a collection of cylinders that is stable for $A$. The result is only partially constructive: they produce a finite set of collections of cylinders and use a counting argument to show that at least one of these collections is stable. The size of the set depends on $|A|$. The representatives coming from these cylinders then get used to replace a given system of equations over $G$ by a finite set of systems of equations in a free group. The solubility of such systems is decidable by Makanin or Razborov \cite{makanin:equations,makanin:decidability,razborov:systems}, and it follows that the solubility of the original system is decidable. (This result has since been generalised to allow torsion \cite{dahmaniguirardel:foliations}.)

The fact that one only has a \emph{set} of collections of cylinders, one of which is stable, the fact that the construction depends on the cardinality of $A$, and the fact that the cylinders are only stable for the particular finite set $A$, all lead to complications in transforming the original system of equations to a system over a free group.

A collection of cylinders on $G$ is \emph{globally stable} if it is stable for the entire group $G$ (see Definition~\ref{def:gsc}). Admitting globally stable cylinders is much stronger than the conclusion of \cite[Thm~4.2]{ripssela:canonical}, and simplifies the algorithm for solving equations. Rips--Sela explicitly ask whether every torsionfree hyperbolic group admits globally stable cylinders \cite[p.508]{ripssela:canonical}, and prove that $C'(\frac18)$ small-cancellation groups do. Much later, Lazarovich--Sageev proved the same for all cubulated hyperbolic groups \cite{lazarovichsageev:globally}, which generalises the result of Rips--Sela because small-cancellation groups are cubulated \cite{wise:cubulating}.

\begin{theorem} \label{mthm:rf_hyp_gsc}
Every residually finite hyperbolic group admits globally stable cylinders.
\end{theorem}

As a consequence of this result, one obtains good representatives of group elements; a \emph{bicombing}, in other words. (In the presence of torsion, one has to use the methods of \cite{dahmani:existential}.) This bicombing is potentially very useful. Indeed, the main ingredient in Mineyev's work \cite{mineyev:straightening} on bounded cohomology is the construction of a weaker structure on hyperbolic groups: a \emph{rational} bicombing satisfying a generalised stability condition. This rational bicombing has seen many applications, including \cite{mineyevmonodshalom:ideal,mineyevyu:baumconnes,yu:hyperbolic,grovesmanning:dehn}. Mineyev reiterates the question of whether hyperbolic groups have a globally stable bicombing \cite[p.808]{mineyev:straightening}.

For instance, in \cite{lazarovich:volume}, Lazarovich proved that hyperbolic groups with globally stable cylinders satisfy \emph{finite-index rigidity}. At the time this was known only for $C'(\frac18)$ groups and cubulated hyperbolic groups. Later, he used Mineyev's rational bicombing to generalise the result to all hyperbolic groups, but the argument is considerably more involved \cite{lazarovich:finite}.

The notion of stable cylinders actually makes sense in any hyperbolic space, and with respect to any isometric action, not just in the Cayley graph of a hyperbolic group with its free, transitive action. Theorem~\ref{mthm:rf_hyp_gsc} is in fact a consequence of a more general result about hyperbolic spaces with certain ``nice'' (i.e. quasimedian, see Definition~\ref{def:sqtr}) equivariant quasiisometric embeddings in finite products of quasitrees; see Theorem~\ref{thm:cylinders}. The case of residually finite hyperbolic groups is then a consequence of \cite{bestvinabrombergfujiwara:proper} and \cite{petyt:mapping}; see Section~\ref{sec:rf_hyp}.

The improved generality of Theorem~\ref{thm:cylinders} enables us to consider hyperbolic spaces other than groups. This includes the curve graph $\C\Sigma$ of a surface $\Sigma$, together with the action of its mapping class group $\MCG\Sigma$.

\begin{theorem} \label{mthm:cs_gsc}
For any finite-type surface $\Sigma$, the pair $(\C\Sigma,\MCG\Sigma)$ admits globally stable cylinders.
\end{theorem}

Theorem~\ref{mthm:cs_gsc} does not immediately follow from Theorem~\ref{thm:cylinders} and known results in the literature. Indeed, although it is known that mapping class groups admit nice equivariant embeddings in finite products of quasitrees \cite{bestvinabrombergfujiwara:proper,petyt:mapping}, it is pointed out on \cite[p.695]{bestvinabrombergfujiwara:proper} that their work does \emph{not} give a quasiisometric embedding of the curve graph in a finite product of quasitrees, only the \emph{thick part} of the curve graph. They leave the existence of equivariant quasiisometric embeddings of curve graphs in finite products of quasitrees an open question.

If one does not require equivariance, then it was noted in \cite{hume:embedding} that curve graphs quasiisometrically embed in finite products of trees. Indeed, they have finite \emph{asymptotic dimension} \cite{bellfujiwara:asymptotic,bestvinabromberg:onasymptotic}, so their Gromov boundaries have finite \emph{capacity dimension} \cite{mackaysisto:embedding}, and hence one can apply work of Buyalo on \emph{hyperbolic cones} to obtain such an embedding \cite{buyalo:capacity}. However, the nature of the hyperbolic cone construction precludes making these embeddings equivariant.

Our third main result resolves this question for curve graphs, by feeding the construction of \cite{bestvinabrombergfujiwara:proper} into the machinery of \emph{dualisable systems} from \cite{petytzalloum:constructing}.

\begin{theorem} \label{mthm:cs_fsqr}
For any finite-type surface $\Sigma$, the curve graph $\C\Sigma$ admits a $\MCG\Sigma$--equivariant, quasimedian, quasiisometric embedding in a finite product of quasitrees.
\end{theorem}

In fact, our arguments again work in more generality than the theorem suggests. The full statement is given as Theorem~\ref{thm:hhg}. The following is a sample application of this more general result, relying on recent work of Tao \cite{tao:property}; see Corollary~\ref{cor:artin}.

\begin{theorem} \label{mthm:artin}
Let $A$ be a residually finite Artin group of large and hyperbolic type. The coned-off Deligne complex $\hat D$, defined in \cite{martinprzytycki:acylindrical}, admits an $A$--equivariant, quasimedian, quasiisometric embedding in a finite product of quasitrees. Moreover, $(A,\hat D)$ admits globally stable cylinders.
\end{theorem}

\subsection*{Key innovations and ideas behind the proof} \label{subsec:new_ideas}

Globally stable cylinders are delicate objects: cylinders are required to \emph{exactly} agree rather than coarsely. A priori, the coarse notion of a $\delta$--hyperbolic space appears too weak to accommodate such a fine-tuned property. Dealing with the coarseness is one of the main challenges. The construction of the cylinders of Theorem~\ref{thm:cylinders} is inspired by \cite{lazarovichsageev:globally}, where Lazarovich--Sageev use fine hyperplane combinatorics and finite-dimensionality to deal with hyperbolic CAT(0) cube complexes. 

To make up for the lack of fine structure, our proof uses two main innovative ingredients. The first is the \emph{generalised Sageev construction}. Developed in \cite{petytsprianozalloum:hyperbolic,petytzalloum:constructing}, this lets one improve certain properties of metric spaces by considering a ``dual'' with respect to some collection of walls. In particular, it can take any hyperbolic graph and produce a thickened space with various fine features. The original Sageev construction \cite{sageev:ends} has stronger requirements on walls and exactly yields a CAT(0) cube complex. Our procedure only yields a median algebra, but allows for more general sets of walls.  The very fact of not yielding a cube complex is a key strength of our approach, as there are residually finite hyperbolic groups with property~(T) that cannot be cubulated. We summarise some of the generalised Sageev construction for hyperbolic spaces in Theorem~\ref{thm:general_fine}. A \emph{wall} on a set $X$ is a bipartition into \emph{halfspaces} $X = h^+ \sqcup h^-$. A collection $\cal Z$ of subsets of $X$ has the (finitary) \emph{Helly property} if every (finite) set of pairwise intersecting members of $\cal Z$ has nonempty total intersection.

\begin{theorem}[{\cite[Thm~7.13]{petytzalloum:constructing}}]\label{thm:general_fine}
Let $S$ be a hyperbolic graph. There is a hyperbolic space with walls $(X,\dist_X,W)$ such that the following hold.
\begin{enumerate}
\item   \textbf{Geometry}: $(X,\dist_W)$ is $\isom(S)$--equivariantly quasiisometric to $S$.
\item   \textbf{Median}: There is a map $\mu:X^3\to X$ making $X$ a median algebra, and $\mu(x,y,z)$ is uniformly close to the coarse centre of a geodesic triangle.
\item   \textbf{Helly}: Metric balls in $X$ satisfy the Helly property. Halfspaces satisfy the finitary Helly property.
\item   \textbf{Gate}: \label{item4:general_fine} If $A\subset X$ is a halfspace or ball, then there exists a 1--Lipschitz retraction $\g_A: X \to A$. 
\item   \textbf{Convexity}: Each halfspace $h^+$ satisfies $\mu(a,b,c) \in h^+$ for all $a,b \in h^+$, $c \in X$.
\end{enumerate}
\end{theorem}

The primary takeaway from Theorem~\ref{thm:general_fine} is that it enables an equivariant transformation of a hyperbolic space $S$ into one with numerous fine properties similar to those of a CAT(0) cube complex. However, this transformation requires a trade-off: the new hyperbolic space $X$ is generally locally infinite. One can think of the above procedure as follows: while many desired properties, such as the Helly property for balls, hold coarsely in $S$, they fail to hold exactly in $S$ due to a lack of sufficient space. This issue is addressed by adding a large number of extra points to $S$, resulting in a locally infinite hyperbolic space $X$ with finer features.

Note that Theorem~\ref{thm:general_fine} does not require residual finiteness. This brings us to the second main innovation in the proof, which uses the embedding theorem of Bestvina--Bromberg--Fujiwara \cite{bestvinabrombergfujiwara:proper}, stating that residually finite hyperbolic groups admit equivariant quasiisometric embeddings in finite products of quasitrees.

The argument of Lazarovich--Sageev in \cite{lazarovichsageev:globally} critically depends on the fact that there is an upper bound on how many hyperplanes can pairwise intersect, namely the dimension of the cube complex. This is not true for the spaces produced by Theorem~\ref{thm:general_fine}. The new idea is to use the Bestvina--Bromberg--Fujiwara embedding to produce a better wallspace structure with coloured walls that, whilst not having a bound on sets of pairwise intersecting walls, does satisfy a kind of surrogate finite-dimensionality that restricts the possible crossing patterns of walls of a given colour. The exact property that we need, namely \emph{strong (QT)}, is abstracted in Definition~\ref{def:sqtr}. This suggests a natural question.

\bsh{Question} \label{question:QT}
Do non--residually-finite hyperbolic groups have strong (QT)?
\esh

A positive answer to Question~\ref{question:QT} would imply that all hyperbolic groups admit globally stable cylinders. 

\medskip

Section~\ref{sec:preliminaries} contains background material and proofs of some preliminary results that will be needed. In Sections~\ref{sec:QT} and~\ref{sec:dualising_subspaces} we construct thickenings of hyperbolic spaces with nice embeddings in products of quasitrees. Section~\ref{sec:stable} contains the proof of Theorem~\ref{thm:cylinders}, which yields globally stable cylinders for such spaces. This is applied to residually finite hyperbolic groups in Section~\ref{sec:rf_hyp}, and to curve graphs, via Theorem~\ref{mthm:cs_fsqr}, in Section~\ref{sec:curve_graph}.

\medskip
\noindent\textbf{Acknowledgments.}
The authors are very grateful to Mladen Bestvina, Nima Hoda, Nir Lazarovich, Michah Sageev, Zlil Sela, and Alessandro Sisto for useful discussions and comments.

\section{Preliminaries} \label{sec:preliminaries}

A \emph{$k$--rough geodesic} in a metric space $X$ is a $(1,k)$--quasiisometrically embedded path in $X$. An \emph{unparametrised} rough geodesic is the image of such an embedding. A metric space $X$ is \emph{$k$--weakly roughly geodesic} if for every $x,y\in X$ and every $r\le\dist(x,y)$, there is some point $z\in X$ such that $\dist(x,z)\ge r-k$ and $\dist(z,y)\ge\dist(x,y)-r-k$, but also $\dist(x,z)+\dist(z,y)\le\dist(x,y)+k$.

\subsection{Coarse medians} \label{subsec:coarse_median}

If $X$ is a roughly geodesic hyperbolic space, then for every triple $x,y,z\in X$ there is a uniformly bounded set that is uniformly close to all sides of every (uniform-quality) rough-geodesic triangle between $x$, $y$, and $z$. A \emph{coarse median} on $X$ is an equivariant choice of point in this bounded set for each triple. This coarse median $\mu:X^3\to X$ satisfies the properties of a median algebra \cite{bandelthedlikova:median} up to a bounded error. Hence, if $Y=\prod_{i=1}^mY_i$ is a direct product of hyperbolic spaces, then the ternary operator $\mu$, given by taking the coarse median in each factor, also satisfies the properties of a median algebra up to finite error.

Bowditch introduced the notion of \emph{coarse median spaces} \cite{bowditch:coarse}, which provides a useful language for working with ternary operators like these. Since we only need a small amount of this language, and only in a few examples, we omit the general definition, as it is slightly technical. However, the natural maps of coarse median spaces will be important for us.

\begin{definition} \label{def:quasimedian}
A map $f:(X,\mu_X)\to(Y,\mu_Y)$ of coarse median spaces is \emph{quasimedian} if there is a constant $\delta$ such that
\[
\dist_Y\big(\mu_Y(fx,fy,fz),\,f(\mu_X(x,y,z))\big) \,\le\, \delta
\]
for all $x,y,z\in X$.
\end{definition}

\subsection{Globally stable cylinders} \label{subsec:cylinders}

The following definition is a minor modification of the definition of globally stable cylinders from \cite{ripssela:canonical} that allows for roughly geodesic spaces. For three points $x,y,z$ in a hyperbolic space, we write $\sgen{y,z}_x$ for their Gromov product, which agrees with the distance from $x$ to the coarse median $\mu(x,y,z)$ up to a uniform additive error.

\begin{definition} \label{def:gsc}
Let $X$ be a $\delta$--roughly geodesic, hyperbolic metric space. For a constant $\theta\ge0$ and points $x,y\in X$, a \emph{$\theta$--cylinder} is a subset $C(x,y)\subset X$ with the property that $\gamma_{xy} \subset C(x,y)\subset\N_\theta(\gamma_{xy})$ for every $\delta$--rough geodesic $\gamma_{xy}$ from $x$ to $y$. 

A choice of $\theta$--cylinder $C(x,y)$ for each $x,y\in X$ is \emph{globally $(k,R)$--stable} if the following two conditions hold.
\begin{enumerate}
\item   \emph{Reversibility:} $C(x,y) = C(y,x)$ for all $x,y\in X$.
\item   \emph{$(k,R)$--stability:} for each $x,y,z\in X$, there exist $R$--balls $B_1,\dots,B_k$ in $X$ such that 
\[
C(x,y)\cap B(x,\sgen{y,z}_x) \ssm \bigcup_{i=1}^k B_i \;=\; C(x,z)\cap B(x,\sgen{y,z}_x) \ssm \bigcup_{i=1}^k B_i.
\]
\end{enumerate}
We say that $X$ \emph{admits globally stable cylinders} if there exists a choice of globally $(k,R)$--stable $\theta$--cylinders for some $k,R,\theta$.

Suppose that a group $G$ acts on $X$. A choice of $\theta$--cylinder $C(x,y)$ for each $x,y\in X$ is \emph{$G$--invariant} if $C(gx,gy)=gC(x,y)$ for all $x,y\in X$ and all $g\in G$. We say that the pair $(G,X)$ admits globally stable cylinders if there exists a choice of $G$--invariant, globally $(k,R)$--stable $\theta$--cylinders for some $k,R,\theta$.

A hyperbolic group $G$ is said to admit globally stable cylinders if it has a Cayley graph $X$ such that $(G,X)$ admits globally stable cylinders.

\begin{figure}[ht]
\centering
\includegraphics[height=5cm, trim = 0mm 4mm 0mm 4mm]{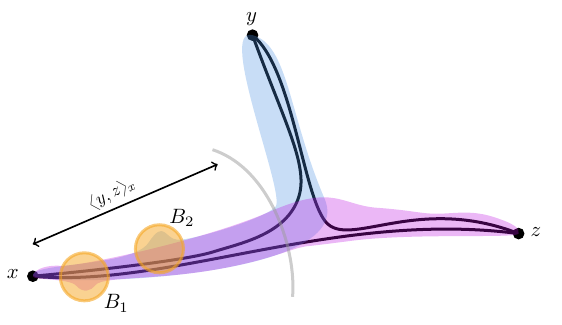}
\caption{The stability condition for cylinders.} \label{fig:cylinders}
\end{figure}
\end{definition}

It turns out that the existence of globally stable cylinders can be passed along equivariant quasiisometries. This is made precise in Proposition~\ref{prop:cylinders_qi_invariant}. For hyperbolic graphs $X$ and $Y$ on which a group $G$ acts properly and cocompactly, it was shown in \cite[Prop.~9]{lazarovichsageev:globally} that $(G,X)$ admits globally stable cylinders if and only if $(G,Y)$ does. Here, though, we are considering metric spaces that are not locally finite graphs: they are locally infinite roughly geodesic spaces. We also allow for arbitrary isometric actions.

We shall use the following two simple statements in the proof of Proposition~\ref{prop:cylinders_qi_invariant}.

\begin{lemma} \label{lem:delta'}
For every $\delta$ there exists $\delta'=\delta'(\delta)$ such that the following holds. Let $X$ be a $\delta$--hyperbolic space and let $x,y,z\in X$. If $\gamma_1$ is a $\delta$--rough geodesic from $x$ to $\mu(x,y,z)$ and $\gamma_2$ is a $\delta$--rough geodesic from $\mu(x,y,z)$ to $y$, then $\gamma_1\cup\gamma_2$ is a $\delta'$--rough geodesic. 
\end{lemma}

\begin{lemma}{{\cite[Lem.~2.9]{petyt:mapping}}} \label{lem:hyp_qi_is_qm}
For each $\delta,\lambda$ there exists $\lambda'$ such that every $\lambda$--quasiisometric embedding of $\delta$--hyperbolic spaces is $\lambda'$--quasimedian.
\end{lemma}

\begin{proposition} \label{prop:cylinders_qi_invariant}
Let $X$ and $Y$ be roughly geodesic hyperbolic spaces and assume that there is a quasiisometry $\phi:Y\to X$. If $X$ has globally stable cylinders, then so does $Y$. 

Suppose that a group $G$ acts on both $X$ and $Y$ and that $\phi$ is $G$--equivariant. If the cylinders on $X$ are $G$--equivariant, then the cylinders on $Y$ can also be taken to be $G$--equivariant.
\end{proposition}

\begin{proof}
Let $C^X(x,y)$ be a globally $(k,R)$--stable choice of $\theta$--cylinders on $X$. Let $\delta$ be such that $X$ and $Y$ are both $\delta$--hyperbolic and $\delta$--roughly geodesic, and let $\delta'$ be as in Lemma~\ref{lem:delta'}. Let $\bar\phi$ be a quasiinverse to $\phi$. Let $\lambda$ be such that $\phi$ and $\bar\phi$ are $\lambda$--quasimedian $\lambda$--quasiisometries with $\dist_Y(\bar\phi\phi(y),y)\le\lambda$ for all $y\in Y$ (see Lemma~\ref{lem:hyp_qi_is_qm}). By the Morse lemma, there is a constant $\kappa=\kappa(\delta,\lambda)$ with the following property: if $x_1,y_1,x_2,y_2\in Y$ have $\dist_Y(x_1,x_2)\le\lambda$ and $\dist_Y(y_1,y_2)\le\lambda$, then for any $(\lambda, \lambda\delta+\lambda)$--quasigeodesic $\gamma$ from $x_2$ to $y_2$, every $\delta'$--rough geodesic from $x_1$ to $y_1$ is at Hausdorff-distance at most $(\kappa-2\lambda)$ from $\gamma$. 

Given $x,y\in Y$, set
\[
C^Y(x,y) \,=\, \N_\kappa\big(\{z\in Y\,:\, \phi(z)\in C^X(\phi(x),\phi(y))\}\big).
\]
We claim that $C^Y(x,y)$ is a globally stable choice of cylinders. Observe that $C^Y(x,y)=C^Y(y,x)$ for all $x,y\in Y$, because the $X$--cylinders are reversible.

\medskip

Let us show that $C^Y(x,y)$ is indeed a cylinder. Since $C^X(\phi(x),\phi(y))$ is a $\theta$--cylinder, it contains a $\delta$--rough geodesic $\gamma$ from $\phi(x)$ to $\phi(y)$, and $\bar\phi\gamma$ is a $(\lambda,\lambda\delta+\lambda)$--quasigeodesic from $\bar\phi\phi(x)$ to $\bar\phi\phi(y)$. Every $z\in Y$ with $\phi(z)\in\gamma$ lies $\lambda$--close to $\bar\phi\gamma$, so the choice of $\kappa$ ensures that every $\delta'$--rough geodesic, hence every $\delta$--rough geodesic, from $x$ to $y$ is contained in $C^Y(x,y)$. Conversely, for each $z\in C^Y(x,y)$, there is some $z'$ with $\phi(z')\in C^X(\phi(x),\phi(y))$ and $d(z,z')\leq \kappa$. Because $C^X(\phi(x),\phi(y))\subset\N_\theta(\gamma)$, there is some $z''\in\gamma$ with $\dist_X(\phi(z'),z'')\le\theta$. It follows that $z'$ lies $(\lambda\theta+2\lambda)$--close to $\bar\phi\gamma$, and hence $z$ is $(2\kappa + \lambda \theta)$--close to every $\delta'$--rough geodesic from $x$ to $y$. We have shown (slightly more than) that $C^Y(x,y)$ is a $(2\kappa+\lambda\theta)$--cylinder.

\medskip

Regarding equivariance, observe that if $g\in G$, then
\[
gC^Y(x,y) \,=\, \N_\kappa\big(\{gz\,:\,g\phi(z)\in gC^X(\phi(x),\phi(y))\}\big) \,=\, C^Y(gx,gy),
\]
because $\phi$ and the $X$--cylinders are assumed to be $G$--equivariant in that statement.

\medskip

It remains to show that the $Y$--cylinders are globally stable. Fix $x,y,z\in Y$, and suppose that $p\in Y$ has $\dist_Y(x,p)\le\sgen{y,z}_x$ but $p\in C^Y(x,y)\ssm C^Y(x,z)$. By definition, there is some $q\in Y$ with $\dist_Y(p,q)\le\kappa$ such that $\phi(q)\in C^X(\phi(x),\phi(y))\ssm C^X(\phi(x),\phi(z))$. To simplify notation, let us write $x'=\phi(x)$, $y'=\phi(y)$, and $z'=\phi(z)$.

If $\dist_X(x',\phi(q))\le\sgen{y',z'}_{x'}$, then since the $X$--cylinders are $(k,R)$--stable, $\phi(q)$ lies in one of a fixed set of $k$ balls of radius $R$ in $X$. In this case, $\bar\phi\phi(q)$ lies in one of a fixed set of $k$ balls of radius $\lambda R+\lambda$ in $Y$, and hence $p$ lies in one of a fixed set of $k$ balls of radius $\kappa+\lambda R+2\lambda$ in $Y$.

Otherwise, $\dist_X(x',\phi(q))>\sgen{y',z'}_{x'}$. We shall use the fact that $\dist_Y(x,p)\le\sgen{y,z}_x$ to show that $q$, and hence $p$, lies uniformly close to $m=\mu_Y(x,y,z)$. This will complete the proof. 

Note that $|\dist_Y(x,m)-\sgen{y,z}_x|\le3\delta$. Let $\gamma_{xm}$ be a $\delta$--rough geodesic from $x$ to $m$ and let $\gamma_{my}$ be a $\delta$--rough geodesic from $m$ to $y$. By Lemma~\ref{lem:delta'}, $\gamma=\gamma_{xm}\cup\gamma_{my}$ is a $\delta'$--rough geodesic. By the cylinder property of $C^Y(x,y)$ shown above, there exists $q_1\in\gamma$ with $\dist_Y(q,q_1)\le 2\kappa+\theta\lambda$. 

If $q_1\in\gamma_{my}$, then the fact that $\dist_Y(x,q_1)\le\sgen{y,z}_x+3\kappa+\theta\lambda$ means that $\dist_Y(m,q_1)\le2\kappa+\theta\lambda+3\delta+2\delta'$, because $\gamma$ is a $\delta'$--rough geodesic. In this case, we have $\dist_Y(p,m)\le6\kappa+2\theta\lambda+3\delta+2\delta'$. 

Otherwise, $q_1\in\gamma_{xm}$. In this case, $\phi(q_1)$ lies on a $(\lambda,\lambda\delta+\lambda)$--quasigeodesic from $x'$ to $\phi(m)$. Because $\phi$ is $\lambda$--quasimedian, the Morse lemma tells us that there is a universal constant $\lambda'=\lambda'(\lambda,\delta)$ such that $\phi(q_1)$ lies $\lambda'$--close to some point $q'_2$ lying on a $\delta$--rough geodesic from $x'$ to $m'=\mu_X(x',y',z')$. Note that $|\dist_X(x',m')-\sgen{y',z'}_{x'}|\le3\delta$. Moreover, we have 
\[
\dist_X(\phi(q),\phi(q_1)) \,\le\, \lambda\dist_Y(q,q_1)+\lambda \,\le\, \lambda(2\kappa+\theta\lambda+1).
\]
So by our current assumptions on $q$, we have 
\[
\dist_X(x',q'_2) \,>\, \sgen{y',z'}_{x'}-\lambda(2\kappa+\theta\lambda+1)-\lambda' \,\ge\, \dist_X(x',m')-\lambda(2\kappa+\theta\lambda+1)-\lambda'-3\delta.
\]
As $q'_2$ lies on a $\delta$--rough geodesic from $x'$ to $m'$, we therefore have 
\[
\dist_X(q'_2,m') \,\le\, \lambda(2\kappa+\theta\lambda+1)+\lambda'+5\delta.
\]
Using the fact that $\bar\phi\phi$ differs from the identity by at most $\lambda$, the triangle inequality now yields
\[
\dist_Y(p,m) \,\le\, \dist_Y(p,q_1) + \lambda + \dist_Y(\bar\phi\phi(q_1),\bar\phi(q'_2)) 
	+ \dist_Y(\bar\phi(q'_2),\bar\phi(m')) + \dist_Y(\bar\phi(m'),\bar\phi\phi(m)) + \lambda,
\]
and we have shown that the latter expression is uniformly bounded in terms of $\delta$, $\lambda$, and $\theta$.

To sum up, we have shown that there is a constant $R'=R'(R,\delta,\lambda,\theta)$ such that $Y$ has globally $(k+1,R')$--stable cylinders, where the additional ball that is removed when considering $x,y,z\in Y$ compared to $\phi(x),\phi(y),\phi(z)$ is centred on $\mu_Y(x,y,z)$.
\end{proof}

\begin{remark} \label{rem:stability}
The above proof shows that, in general, if $X$ has globally $(k,R)$--stable cylinders, then $Y$ has globally $(k+1,R')$--stable cylinders. However, if in $X$ it happens that, for any $x,y,z\in X$, one of the balls $B_i$ that is removed is centred on the median $\mu_X(x,y,z)$, then the proof shows the stronger statement that $Y$ has globally $(k,R')$--stable cylinders.
\end{remark}

\subsection{Dualisable systems} \label{subsec:duals}

The material in this section comes from \cite{petytzalloum:constructing}. 

A \emph{set with walls} is a pair $(S,W)$, where $S$ is a set and $W$ is a set of bipartitions $h=\{h^+,h^-\}$ of $S$. That is, $h^-,h^+\subset S$ have $S=h^-\cup h^+$ and $h^-\cap h^+=\varnothing$. We refer to $h$ as a \emph{wall}, and to $h^\pm$ as the \emph{halfspaces} of $h$.

\begin{definition}[Ultrafilter] \label{def:filter}
A \emph{filter} $\phi$ on $W$ consists of a subset $W'\subset W$ and a choice of halfspace $\phi(h)\in\{h^+,h^-\}$ for each $h\in W'$ such that: 
\[
\text{if }h_1,h_2\in W'\text{ have }h_1^+\subset h_2^+\subset S,\text{ then }\phi(h_1)=h_1^+\text{ implies that }\phi(h_2)=h_2^+.
\]
We say that $\phi$ is \emph{supported} on $W'$. An \emph{ultrafilter} is a filter whose support is $W$. 
\end{definition}

Each $s\in S$ determines an ultrafilter $\phi_s$ by setting $\phi_s(h)$ to be the halfspace containing $s$, for each $h\in W$. Let $\hat S$ be the set of all ultrafilters on $W$. Although it is not in general injective, we view the map $S\to\hat S$ given by $s\mapsto\phi_s$ as an inclusion map of $S$ in $\hat S$. We say that $h_1,h_2\in W$ \emph{cross} if all four orientations of $h_1$ and $h_2$ are filters; that is, if $S$ meets all four quarterspaces.

The set $\hat S$ is equipped with a ternary operation $\mu$ that makes it a median algebra. Indeed, given $\phi_1,\phi_2,\phi_3\in\hat S$, for each $h\in W$ we set $\mu(\phi_1,\phi_2,\phi_3)(h)$ to be the halfspace of $h$ selected by the majority of the $\phi_i$.

We say that $c\subset W$ \emph{separates} $x\in\hat{S}$ from $y\in\hat S$ if $x(h)\ne y(h)$ for all $h\in c$. That is, $x$ and $y$ orient the elements of $c$ oppositely. A sequence $(h_i)$ of walls is called a \emph{chain} if each $h_i$ separates $h_{i-1}$ from $h_{i+1}$.

Given a collection $\D$ of subsets of $W$ (not necessarily a collection of chains), let $\dist_{\D}$ be the function on $\hat S\times\hat S$ given by 
\[
\dist_\D(x,y)=\sup\{|c|:c\in\D\text{ separates }x\text{ from }y\},
\]
which takes values in $\mathbf N\cup\{\infty\}$. The following is a key definition in \cite{petytzalloum:constructing}.

\begin{definition}[Dualisable system, dual space] \label{def:dualisable_system} 
Let $(S,W)$ be a set with walls, and let $\D\subset2^W$ be closed under taking subsets, with $\{h\}\in\D$ for each $h\in W$. We say that $\D$ is a \emph{dualisable system} on $W$ if $\dist_\D(\phi_s,\phi_t)<\infty$ for each $s,t\in S$. 

The \emph{$\D$--dual} of $S$ is the set $S_\D=\{x\in\hat S\,:\dist_\D(x,\phi_s)<\infty\text{ for all }s\in S\}$. It is equipped with the function $\dist_\D$ and the ternary operation $\mu$, which are, respectively, a metric and a median \cite[Lem.~3.4, 3.6]{petytzalloum:constructing}.
\end{definition}

If every element of a dualisable system $\D$ is a chain, then we refer to $\D$ as a dualisable system of chains.

The following is a strong form of convexity for subsets of $S_\D$; see \cite[Lem.~3.11, 3.13]{petytzalloum:constructing}.

\begin{definition} \label{def:gated}
A nonempty subset $A\subset S_\D$ is \emph{gated} if there is a set $H$ of halfspaces of elements of $W$ such that $A=S_\D\cap\bigcap_{h^+\in H}h^+$. Equivalently, there is a filter $\psi$ supported on some $W'\subset W$ such that $A=\{\phi\in S_\D\,:\,\phi|_{W'}=\psi\}$.
\end{definition}

Every gated subset $A\subset S_\D$ comes with a \emph{gate map} $\g_A:S_\D\to A$. Given $x\in S_\D$, the ultrafilter $\g_A(x)$ is obtained from $x$ by switching the orientation of exactly the walls separating $x$ from $A$. The following two lemmas give a natural source of examples of gated subsets.

\begin{lemma}[{\cite[Lem.~4.2]{petytzalloum:constructing}}] \label{lem:balls_gated}
If $\D$ is a dualisable system of chains, then every ball in $S_\C$ is gated.
\end{lemma}

\begin{lemma} \label{lem:gating_halfspaces}
Let $A\subset S_\D$ be a gated subset. For any halfspace $h^-$, the gate $\g_A(h^-)$ is gated.
\end{lemma}

\begin{proof}
By definition, there is a set $U\subset W$ and a filter $\phi$ on $U$ such that $A=S_\D\cap\bigcap_{k\in U}\phi(k)$. Let $V$ be the set of all walls not in $U$ that have a halfspace containing $h^-$. For $k\in V$, set $\phi(k)$ to be that halfspace. One can check that $\phi$ is a filter on $U\cup V$, and $\g_A(h^-)=S_\D\cap\bigcap_{k\in U\cup V}\phi(k)$.
\end{proof}

The following are useful properties for a dualisable system to satisfy.

\begin{definition} \label{def:gluable_separated}
Let $\D$ be a dualisable system of chains.
\begin{itemize}
\item 	$\D$ is \emph{$m$--gluable} if, whenever $c_1,c_2\in\D$ are such that $c_1\cup c_2$ is a chain with $c_1\subset c_2^-$, there is some $d\subset c_1\cup c_2$ with $|d|\leq m$ such that $(c_1\cup c_2)\ssm d\in \D$.
\item 	$\D$ is \emph{$L$--separated} if for any $\{h_1,h_2\} \in \D$  and $c \in \D$, if every member of $c$ crosses both $h_1$ and $h_2$, then $|c| \leq L$.
\end{itemize}
\end{definition}

(Although there is a minor difference between the notion of gluability above and the one in \cite{petytzalloum:constructing}, this does not affect any arguments. The definition here has slightly more flexibility.) These notions provide a way to build hyperbolic spaces.

\begin{proposition}[{\cite[Cor.~5.7]{petytzalloum:constructing}}] \label{prop:dual_hyp}
If $\D$ is a separated, gluable dualisable system of chains on a set with walls $(S,W)$, then the dual space $S_\D$ is a roughly geodesic hyperbolic space.
\end{proposition}

The following strengthens Proposition~\ref{prop:dual_hyp} in the case where the system is $0$--separated.

\begin{proposition} \label{prop:0-sep_quasitree}
If $\D$ is a 0--separated, $m$--gluable dualisable system of chains on a set with walls $(S,W)$, then the dual space $S_\D$ is a quasitree.
\end{proposition}

\begin{proof}
By Proposition~\ref{prop:dual_hyp}, $S_\D$ is roughly geodesic hyperbolic space, so by \cite[Prop.~A.2]{petytsprianozalloum:hyperbolic} it is coarsely dense in a geodesic hyperbolic space. We shall use the following reformulation of Manning's bottleneck criterion to show that $S_\C$ is a quasitree; see \cite[Thm~4.6]{manning:geometry} for the original, and \cite[\S3.6]{bestvinabrombergfujiwara:constructing} for the reformulation.
\begin{displayquote}
A geodesic space $X$ is a quasitree if and only if there is some $\delta$ such that for any $x,y\in X$ and any path $\gamma$ from $x$ to $y$, every point on a geodesic from $x$ to $y$ lies $\delta$--close to $\gamma$. 
\end{displayquote}

Given $x,y\in S_\C$, let $\sigma_{xy}$ be a median, $3m$--roughly geodesic path in $S_\C$ from $x$ to $y$, the existence of which is given by \cite[Prop.~4.6]{petytzalloum:constructing}. By the Morse lemma and the coarse density of $S_\C$, it suffices to show that, for any $x,y\in S_\C$ and any path $\gamma\subset S_\C$ from $x$ to $y$, every point on $\sigma_{xy}$ is uniformly close to $\gamma$.

Given $z\in\sigma_{xy}$, consider an element $d\in\C$ realising $\dist_\C(x,y)$, and let $h_1,h_2\in W$ be the adjacent elements of $d$ such that $z\in h_1^+\cap h_2^-$. Let $\bar\gamma$ be the subpath of $\gamma$ contained in $h_1^+\cap h_2^-$, and let $w\in\bar\gamma$ be a point minimising $\dist_\C(z,w)$.

Suppose that $\dist_\C(z,w)\ge 3m+5$, and let $c\in\C$ realise $\dist_\C(z,w)$. Since $\sigma_{xy}$ is a median path, no element of $c$ can separate $z$ from both $h_1$ and $h_2$. Since $\C$ is 0--separated, at most one element of $c$ can cross each $h_i$. The remaining elements of $c$ must either: separate $h_1$ and $w$ from $h_2$ and $z$; or separate $h_2$ and $w$ from $h_1$ and $z$; or separate $w$ from $h_1$, $h_2$, and $z$. Thus, since $|c|\ge3m+5$, it must either have at least $2m+1$ elements separating $h_1$ from $h_2$, or at least $m+3$ elements separating $w$ from both $h_1$ and $h_2$.

In the first case, applying gluability of $\C$ twice, we get a subset of $d'\subset d\cup c$ with $d'\in\C$ and $|d'|\ge|d|+1$, which contradicts the maximality of $d$.

In the second case, depicted in Figure~\ref{fig:quasitree}, write $c=(k_1,\dots,k_p)$, with $k_i$ separating $w\in k_i^-$ from $k_{i+1}$. If $k_{m+3}$ does not cross $h_1$, then $k_{m+3}^-\subset h_1^+$, and there exists a last point $w'\in\bar\gamma$ such that $w'\in h_1^+\cap k_{m+3}^+$. Let $c'=(l_1,\dots,l_q)\in\C$ realise $\dist_\C(z,w')$. By the choice of $w$, we have $q\ge p$. By the choice of $w'$, at most one element of $c'$ can separate $w'$ from $k_{m+3}$. Because $\C$ is 0--separated, at most one element of $c'$ can cross $k_{m+3}$. Hence $(l_1,\dots,l_{q-2})$ all separate $z$ from $k_{m+3}$. As $\C$ is $m$--gluable, we find a subchain of $c\cup c'$ of length at least $(q-2)+(m+3)-m>q\ge p$ that is in $\C$ and separates $z$ from $w$. This contradicts the maximality of $c$, and hence $k_{m+3}$ crosses $h_1$.

A symmetric argument shows that $k_{m+3}$ also crosses $h_2$. This contradicts the 0--separation of $\C$. Hence $z$ is at distance at most $3m+4$ from $\gamma$.
\end{proof}

\begin{figure}[ht]
\centering
\includegraphics[height=6.5cm, trim = 0mm 4mm 0mm 4mm]{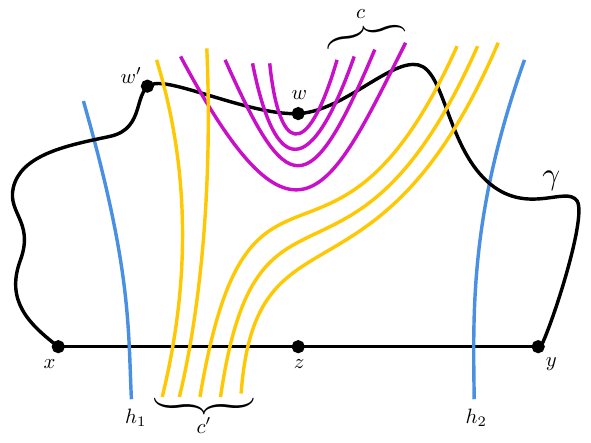}
\caption{The case in the proof of Proposition~\ref{prop:0-sep_quasitree} where many elements of $c$ separate $w$ from $h_1$ and $h_2$.} \label{fig:quasitree}
\end{figure}

The space $S_\D$ is \emph{a priori} much larger than $(S,\dist_\D)$. The following gives a criterion for this to not be the case.

\begin{proposition}[{\cite[Prop.~5.8]{petytzalloum:constructing}}] \label{prop:density}
Let $\D$ be an $L$--separated, $m$--gluable system of chains. If $(S,\dist_\D)$ is $k$--weakly roughly geodesic, then $S$ is $(3k+4(L+m+1))$--coarsely dense in $S_\D$.
\end{proposition}

\subsection{Large-scale geometry of mapping class groups} \label{subsec:mcg}

Here we discuss some aspects of the geometry of mapping class groups. Their primary use will be to prove Proposition~\ref{prop:curve_graph}, so the reader who is prepared to take that on faith can skip this section. They are all fairly standard from the perspective of \emph{hierarchical hyperbolicity}, so we refer the reader to \cite{behrstockhagensisto:hierarchically:2} or \cite{casalsruizhagenkazachkov:real} for broader and more detailed treatments.

Let $\Sigma$ be a finite-type surface, with curve graph $\C\Sigma$ and mapping class group $\MCG\Sigma$. The graph metric on the hyperbolic space $\C\Sigma$ is denoted $\dist_\Sigma$. There is an equivariant, coarsely Lipschitz map $\pi:\MCG\Sigma\to\C\Sigma$. More generally, if $U$ is an essential, non-pants subsurface of $\Sigma$, then there is a coarsely Lipschitz map $\pi_U:\MCG\Sigma\to\C U$. Let $\s$ be the set of isotopy classes of such subsurfaces.

The mapping class group $\MCG\Sigma$ is also a coarse median space \cite{bowditch:coarse}, with coarse median given by the ``centroid'' construction of Behrstock--Minsky \cite{behrstockminsky:centroids}. More precisely, given $x_1,x_2,x_3\in\MCG\Sigma$, the coarse median $\mu(x_1,x_2,x_3)$ is a mapping class such that for every $U\in\s$ we have that $\pi_U\mu(x_1,x_2,x_3)$ lies uniformly close to the coarse median in $\C U$ of the $\pi_Ux_i$. In particular, the maps $\pi_U$ are all uniformly quasimedian.

There is also a notion of \emph{rank} for coarse median spaces \cite{bowditch:coarse}. Unbounded hyperbolic spaces have rank 1, and the direct product of $m$ unbounded hyperbolic spaces has rank $m$. For groups that are coarse median spaces, the rank is equal to the maximal dimension of a quasiflat \cite[Prop.~3.1]{munropetyt:obstructions}, which for $\MCG\Sigma$ is the complexity $\xi$ of $\Sigma$. This goes into the definition of \emph{hulls} of subsets \cite[\S6]{bowditch:convex}.

\begin{definition}[Median-quasiconvex; hull]
A subset $A\subset\MCG\Sigma$ is \emph{$\delta$--median-quasiconvex} if, for any $a,b\in A$ and any $x\in\MCG\Sigma$, the point $\mu(a,b,x)$ is $\delta$--close to $A$.

Given $B\subset\MCG\Sigma$, let $\cal J(B)=\{\mu(b,b',x)\,:\,b,b'\in B,\,x\in\MCG\Sigma\}$. Iterating, we define $\hull B=\cal J^\xi(B)$, which is uniformly median-quasiconvex \cite[Lem.~6.1]{bowditch:convex}.
\end{definition}

If $A\subset\MCG\Sigma$ is median-quasiconvex, then $\pi_U(A)$ must be quasiconvex for every $U\in\s$. This gives a \emph{coarse gate} map $\p_A:\MCG\Sigma\to A$, similarly to the construction of the coarse median: given $x\in\MCG\Sigma$, the point $\p_A(x)$ is such that $\pi_U\p_A(x)$ is uniformly close to the closest-point projection to $\pi_U(A)$ of $\pi_U(x)$, for all $U$.

One natural family of median-quasiconvex subsets is given by curve stabilisers. Given a curve $\alpha$ on $\Sigma$, the stabiliser of $\alpha$ is quasimedian quasiisometric to a product $\mathbf P_\alpha$, with the two factors corresponding to $\alpha$ and its complementary subsurface. See \cite[\S5]{behrstockhagensisto:hierarchically:2} or \cite[\S17]{casalsruizhagenkazachkov:real} for more detail.

\begin{remark} \label{rem:E}
There are various constants involved in the above discussion, but we fix a sufficiently large value of $E$ compared to them; see \cite[Ren.~1.6]{behrstockhagensisto:hierarchically:2}. Thus $\C\Sigma$ is $E$--hyperbolic, $\pi:\MCG\Sigma\to\C\Sigma$ is $E$--quasimedian and $E$--coarsely Lipschitz, and hulls are $E$--median-quasiconvex.
\end{remark}

\section{Dualising quasitrees} \label{sec:QT}

The goal of this section is to replace a given quasitree by one with better median properties.

Let $(T,\dist)$ be a $\delta$--hyperbolic quasitree. Choose $K>100\delta$ large enough so that every ball $B$ of radius $K$ disconnects $T$. Let $\{C_i^B\}_{i \in I}$ be the components of $T\ssm B$. Each $C^B_i$ defines two natural bipartitions of $T$, namely $(C^B_i,\, T\ssm C^B_i)$ and $(C^B_i\cup B,\, T\ssm(C^B_i\cup B))$. Let $W$ be the collection of all such bipartitions of $T$ over all balls of radius $K$ and all components of their complements.

We say that a set of balls is \emph{disparate} if for each pair the distance between their centres is at least $10K$. Let $c=\{h_1,\dots,h_n\}$ be a chain of walls, with $h_i$ defined by some $K$--ball $B_i$. If $\{B_i\}$ is disparate, then we say that $c$ is disparate. We define
\[
\D \,=\, \{\text{disparate chains of walls in }W\}.
\]

\begin{lemma} \label{lem:disparate_properties}
$\D$ is a 1--gluable, 0--separated, dualisable system of chains on $(T,W)$. Moreover, $(T, \dist_{\D})$ is quasiisometric to $(T,\dist)$.
\end{lemma}

\begin{proof}
$\D$ contains all singletons and is closed under taking subsets. It is dualisable because any element of $\D$ separating $s\in S$ from $t\in S$ has cardinality at most $\lceil\frac1{10K}\dist(s,t)\rceil$. Conversely, given $s$ and $t$, if we take balls with centres on a geodesic from $s$ to $t$ at pairwise distance $20K$, then we see that $\dist_\D(s,t)\ge\frac{\dist(s,t)}{20K}-1$. The 1--gluability of $\D$ is clear, and 0--separation holds because the balls defining $W$ all have diameter at most $2K$.
\end{proof}

From the dualisable system $\D$ on $(T,W)$, we obtain a dual space $T_\D$ as described in Section~\ref{sec:preliminaries}. Already from Lemma~\ref{lem:disparate_properties} and Proposition~\ref{prop:dual_hyp} we have that $T_\D$ is a roughly geodesic hyperbolic space, but this is strengthened by the following, which in particular shows that $T_\D$ is a quasitree and the inclusion of $T$ into $T_\D$ is quasimedian.

\begin{lemma} \label{lem:quasitree_dense}
$(T, \dist_{\D})$ is 11--dense in its dual $T_\D$. Any group acting isometrically on $(T,\dist)$ acts isometrically on $T_\D$.
\end{lemma}

\begin{proof} 
It is easy to see that $(T,\dist_\D)$ is 1-roughly geodesic. Lemma~\ref{lem:disparate_properties} and Proposition~\ref{prop:density} together yield the density. Actions on $T$ pass to actions on $T_\D$ because $W$ is preserved by isometries.
\end{proof}

\section{Dualising subspaces of products} \label{sec:dualising_subspaces}

In this section, we build on the construction of Section~\ref{sec:QT} to show how to produce, for certain hyperbolic spaces, a thickening with improved median properties. The hyperbolic spaces we consider are singled out by the following.

\begin{definition}[Strong (QT)] \label{def:sqtr}
Let $X$ be a roughly geodesic coarse median space, and let $G$ be a group acting on $X$. The \emph{strong quasitree rank} of $(X,G)$ is the infimal value of $m$ such that there exists a $G$--equivariant, quasimedian, quasiisometric embedding of $X$ in a product of $m$ quasitrees. If $(X,G)$ has finite strong quasitree rank, then we say that it has \emph{strong (QT)}.
\end{definition}

When the group $G$ is understood we simply refer to the strong quasitree rank of $X$, or to $X$ having strong (QT). The adjective ``strong'' indicates that the map is quasimedian; for groups this is the difference between strong (QT) and the property (QT) introduced in \cite{bestvinabrombergfujiwara:proper}.

Throughout this section, we shall fix a roughly geodesic hyperbolic space $X$, acted on by a group $G$, with strong quasitree rank $m<\infty$. We fix a finite product $\prod_{i=1}^mT_i$ of quasitrees into which $X$ admits a $G$--equivariant, quasimedian, quasiisometric embedding. We consider the product as being equipped with the $\ell^1$ metric. 

For each $i$, let $W_i$ and $\D_i$ be as constructed in Section~\ref{sec:QT}, with corresponding dual spaces $T_{\D_i}$. By Lemma~\ref{lem:quasitree_dense}, $T_i$ is coarsely dense in $T_{\D_i}$ and we have an action of $G$ on $\prod_{i=1}^mT_{\D_i}$. Let $S$ denote the image of $X$ in $\prod_{i=1}^mT_{\D_i}$.

\begin{lemma} \label{lem:orbit_hyperbolic}
$S$ is a roughly geodesic hyperbolic space.
\end{lemma}

\begin{proof}
Since $X\to S$ is a quasiisometry and $X$ is hyperbolic, it suffices to show that $S$ is roughly geodesic. Given $s,t\in S$, let $\gamma$ be the image in $S$ of a rough geodesic in $X$ with endpoints mapping to $s$ and $t$. Let $\pi_i$ denote the projection map $\prod_{j=1}^mT_{\D_j}\to T_{\D_i}$. By the quasimedian assumption, for each $i$ the path $\pi_i\gamma$ is a quasimedian path. Since $T_{\D_i}$ is a quasitree by Lemma~\ref{lem:quasitree_dense}, this means that $\pi_i\gamma$ is in fact an unparametrised rough geodesic. As this holds for every $i$, it follows that $\gamma$ is an unparametrised rough geodesic.    
\end{proof}

\ubsh{Walls}
The elements of $W_i$ induce walls on $S$ by considering the projection to $T_{\D_i}$. Let $W$ be the set of such walls on $S$, with the crossing relation on $W$ defined so that two walls in $W$ cross exactly when $S$ meets all four quarterspaces. 

To be clear, two walls in $W$ can only cross if the corresponding bipartitions of $\prod_{i=1}^mT_i$ cross; i.e. they either come from distinct factors or from two nearby balls in the same factor. However, that is \emph{not} a sufficient condition for two walls to cross in $W$. Indeed, there will in general be walls coming from distinct $W_i$ that do not cross in $W$, because $S$ does not meet all four quarterspaces.
\uesh

\begin{definition}[Grid]
A \emph{grid} in $W$ is the data of two chains $d_1 \in \D_i, d_2 \in \D_j$, possibly with $i=j$, such that, when viewed as elements of $W$, each member of $d_1$ crosses each member of $d_2$. That is, they cross in $S$.
\end{definition}

The next lemma shows that $W$ cannot have large square grids.

\begin{lemma} \label{lem:no_grids}
There exists $L$ such that if $d_1,d_2$ form a grid in $W$, then $\min\{|d_1|, |d_2|\} \leq L$.
\end{lemma}

\begin{proof} 
Suppose that $d_1=\{h_1, \dots, h_m\}$ and $d_2=\{k_1, \dots, k_n\}$ form a grid in $W$. Let $s_1,s_2,s_3,s_4\in S$ be points in $h_1^-\cap k_1^-,\, h_m^+\cap k_1^-,\, h_m^+ \cap k_n^+$, and $h_1^- \cap k_n^+$, respectively, and let $\gamma_i$ be a rough geodesic in $S$ from $s_i$ to $s_{i+1}$. Since the projection of $\gamma_i$ to each $T_{\D_i}$ is an unparametrised rough geodesic, we find that $\gamma_1$ lies in a uniform neighbourhood of $k_1^-$, and similarly for the other $\gamma_i$. The $\gamma_i$ therefore form a rough geodesic quadrilateral whose thickness is lower bounded in terms of $\min\{|d_1|,|d_2|\}$. The existence of $L$ as in the statement follows from the fact that $S$ is hyperbolic.
\end{proof}



We next use the collection $W$ of walls on $S$ to define a dualisable system $\C$ that will be used to obtain our thickening of the space $X$. 

\ubsh{$L$--chains and the dualisable system $\C$}
First define 
\[
\D^1 \,=\, \big\{\bigcup_{i=1}^m d_i \,:\, d_i \in \D_i\big\} \,\subset\, 2^W.
\]
Note that, by definition, $\varnothing\in\D_i$ for each $i$, so an element of $\D^1$ may be comprised only of walls from a subset of the $T_i$. Observe that $\dist_{\D^1}$ is exactly the $\ell^1$ metric on $\prod_{i=1}^mT_{\D_i}$. However, we are considering elements of $\D^1$ with the crossing relation of $W$. Now set
\[
\D \,=\, \{d\in\D^1\,:\,d\text{ is a chain}\}.
\]
In particular, for each $i$ we have $\D_i\subset\D$.

Let $L$ be as in Lemma~\ref{lem:no_grids}. We say that an element $d\in\D$ is an \emph{$L$--chain} if for any $h,k\in d$, any $d'\in\D$ whose elements all cross both $h$ and $k$ has $|d'|\le L$. Define
\[
\C \,=\, \{d\in\D\,:\,d\text{ is an }L\text{--chain}\}.
\]
\uesh

\begin{lemma} \label{lem:C_is_gluable}
$\C$ and $\D$ are $m$--gluable dualisable systems, and $\C$ is $L$--separated.
\end{lemma}

\begin{proof}
$\C$ and $\D$ contain all singletons and are closed under taking subsets. Since they are subsets of $\D^1$, the fact that $\dist_{\D^1}$ agrees with the $\ell^1$ metric on $\prod_{i=1}^mT_{\D_i}$ shows that $\C$ and $\D$ are dualisable systems. The subset $\C\subset\D$ consists of all $L$--chains, so in particular $\C$ is an $L$--separated dualisable system.

It remains to show that $\C$ and $\D$ are gluable. Let $c_1=\{h_1,\dots,h_p\}$ and $c_2=\{h'_1,\dots,h'_q\}$ be elements of $\D$ such that $c=c_1\cup c_2$ is a chain, with $c_1\subset c_2^-$ and $c_2\subset c_1^+$. We can write $c_1=\bigcup_{i=1}^md_i$ and $c_2=\bigcup_{i=1}^md'_i$ for some $d_i,d'_i\in\D_i$. Since each $\D_i$ is 1--gluable, we can remove a single element from each $d_i\cup d'_i$ to obtain a subset $c'\subset c$ with $c'\in\D^1$. In particular, at least one of $h_p$ and $h'_1$ does not lie in $c'$. Since $c$ is a chain, we moreover have that $c'\in\D$. This shows that $\D$ is $m$--gluable.

Now suppose $c_1,c_2\in\C$. Since at least one of $h_p$ and $h'_1$ is not in $c'$, any $d\in\D$ that crosses two elements of $c'$ must either: cross two elements of some $c'_1$; or cross $h_{p-1}$ and $h'_1$; or cross $h_p$ and $h'_2$. Since $c$ is a chain, in either of the latter two cases we have that $d$ crosses two elements of some $c_i$. Since each $c_i$ is an $L$--chain, this shows that $|d|\le L$, so $c'\in\C$. Thus $\C$ is also $m$--gluable.
\end{proof}

The thickening of $X$ that we shall consider will be the dual space $S_\C$.

\begin{proposition} \label{prop:dual_SC}
$S_\C$ is a $3m$--roughly geodesic hyperbolic space that is $G$--equivariantly quasiisometric to $X$.
\end{proposition}

\begin{proof}
The fact that $S_\C$ is a $3m$--roughly geodesic hyperbolic space is a combination of Lemma~\ref{lem:C_is_gluable} and Proposition~\ref{prop:dual_hyp}. Moreover, $G$ acts on $S_\C$ because it acts on $S$. We must show that $S_\C$ is quasiisometric to $X$. We do this in two steps. First, we show that, when restricted to $S$, the metric $\dist_\C$ is quasiisometric to the subspace metric $\dist_S$ inherited from $\prod_{i=1}^mT_{\D_i}$. Then we show that $(S,\dist_\C)$ is coarsely dense in $S_\C$.

Since $\dist_S=\dist_{\D^1}$ on $S$, the fact that $\C\subset\D^1$ implies that $\dist_\C\le\dist_S$. For the other direction, let $s,t\in S$ and consider some $d=(d_i)_{i=1}^m\in\D^1$ realising $\dist_S(s,t)$. There is some $i$ such that $|d_i|\ge\frac1m|d|$. Note that $d_i$ is a chain. Write $d_i=(h_1,\dots,h_n)$, and consider the subchain $d'_i=(h_{2L},h_{4L},h_{6L},\dots)$. We have $|d'_i|\ge\frac{|d_i|}{2L}-1$. Any element of $W$ crossing two elements of $d'_i$ must cross a subchain of $d_i$ of length $2L>L$. By Lemma~\ref{lem:no_grids}, any chain of such elements has length at most $L$, which shows that $d'_i$ is an $L$--chain. We have shown that $\dist_\C\ge\frac1{2mL}(\dist_S-m)$. Thus $(S,\dist_\C)$ is quasiisometric to $(S,\dist_S)$.

It remains to show that $(S,\dist_\C)$ is coarsely dense in $S_\C$. According to Proposition~\ref{prop:density}, it is enough to show that $(S,\dist_\C)$ is weakly roughly geodesic.

Let $s,t\in S$, and let $\gamma\subset S$ be a rough geodesic from $s$ to $t$, the existence of which is given by Lemma~\ref{lem:orbit_hyperbolic}. Let $c\in\C$ realise $\dist_\C(s,t)$, and write $c=\{h_1,\dots,h_n\}$, where $n=\dist_\C(s,t)$. Since $c$ is a chain and $\gamma$ is a rough geodesic, for each $r\le n$ there is a point $p_r\in\gamma$ lying uniformly close to $h_r^+\cap h_{r+1}^-\subset S$. That is, there is a uniform constant $\eps$ such that all of $h_1,\dots,h_{r-\eps}$ separate $s$ from $p_r$, and all of $h_{r+1+\eps},\dots,h_n$ separate $p_r$ from $t$. In particular, $\dist_\C(s,p_r)\ge r-\eps$ and $\dist_\C(p_r,t)\ge n-r-\eps$.

Let $c_1\in\C$ realise $\dist_\C(s,p_r)$, and let $c_2\in\C$ realise $\dist_\C(p_r,t)$. According to \cite[Lem.~5.3]{petytzalloum:constructing}, there is some $c'\subset c_1\cup c_2$ such that $c'\in \C$ separates $s$ from $t$ and $|c'|\ge|c_1|+|c_2|-L-m-1$. By the choice of $|c|$, we have $|c'|\le|c|$, which shows that $\dist_\C(s,p_r)+\dist_\C(p_r,t)\le\dist_\C(s,t)+L+m+1$. We have shown that the points $p_1,\dots,p_n$ form a uniform-quality weak rough geodesic from $s$ to $t$, completing the proof.
\end{proof}

\section{Cylinders} \label{sec:stable}

Let $X$ be a roughly geodesic hyperbolic space, acted on by a group $G$, with strong quasitree rank $m<\infty$. In this section, we prove that the thickening $S_\C$ constructed in Section~\ref{sec:dualising_subspaces} has globally stable cylinders. Our proof is inspired by that of \cite{lazarovichsageev:globally}. We maintain the notation of Section~\ref{sec:dualising_subspaces}.

We shall have cause to consider elements of $\C$ that arise from a single $T_i$.

\begin{definition}[Monochromatic]
We say that $c\in\C$ is \emph{monochromatic} if there exists $i$, the \emph{colour} of $c$, such that $c$ is an element of $\D_i$.
\end{definition}

Given a pair of points $x,y\in S_\C$, we shall consider certain collections of halfspaces, which will be used to define ``intervals'' from $x$ to $y$.

\ubsh{Non-separating walls}
For $x,y\in S_\C$, let $\H(x,y)$ denote the set of all walls $h\in W$ that do not separate $x$ from $y$. We shall always orient $h\in\H(x,y)$ so that $\{x,y\}\subset h^+$. Define
\[
[x,y] \,=\, \bigcap_{h\in\H(x,y)}h^+.
\]
\uesh

Note that, as an intersection of halfspaces, the set $[x,y]\subset S_\C$ is gated.

\begin{lemma} \label{lem:interval_quasiline}
For any $x,y\in S_\C$, the set $[x,y]$ contains a $3m$--rough geodesic from $x$ to $y$ and is uniformly quasiisometric to an interval of length $\dist_\C(x,y)$.
\end{lemma}

\begin{proof}
By the Morse lemma, it suffices to show that every point of $[x,y]$ lies on a uniform quasigeodesic from $x$ to $y$ inside $[x,y]$. Given $z\in[x,y]$, there exists a coarsely connected median path from $x$ to $y$ that passes through $z$ by \cite[Prop.~4.6]{petytzalloum:constructing}: it is the concatenation of a median path from $x$ to $z$ with a median path from $z$ to $y$. Taking $z=x$, we have a $3m$--rough geodesic in $[x,y]$ from $x$ to $y$. Otherwise, as any coarsely connected quasimedian path in a hyperbolic space is an unparametrised quasigeodesic \cite[Lem.~2.16]{petyt:onlarge}, a change of parametrisation proves the lemma.
\end{proof}

Recall the constant $L$ from Lemma~\ref{lem:no_grids}, which bounds the size of square grids in $W$. We are free to assume that $L\ge3$.

\ubsh{Distant walls}
Let $x,y\in S_\C$. We say that $h\in\H(x,y)$ is \emph{distant} (from $x$ and $y$) if every monochromatic $c\in\C$ that separates $x$ from $y$ and crosses $h$ has $|c|\le L$. Write $\H^L(x,y)\subset\H(x,y)$ for the set of all distant walls. Define
\[
I(x,y) \,=\, \bigcap_{h\in\H^L(x,y)}h^+.
\]
\uesh

The cylinders we shall consider will be uniform neighbourhoods of the subsets $I(x,y)$; taking a neighbourhood is only necessary because $I(x,y)$ may not contain all rough geodesics from $x$ to $y$. The other inclusion, that cylinders lie in a neighbourhood of every rough geodesic, will be a consequence of the following. 

\begin{lemma} \label{lem:cylinder_close}
For every $x,y\in S_{\C}$, we have $[x,y] \subset I(x,y) \subset \cal N([x,y], 1)$.
\end{lemma}

\begin{proof} 
Given $p\in I(x,y)$, let $q=\g_{[x,y]}(p)$. If $\dist_\C(p,q)>1$ then there exists $\{h_1,h_2\} \in \C$ separating $p$ from $q$. By definition of the gate, each $h_i$ separates $p$ from $[x,y]$. In particular, $h_i\in\H(x,y)$. After relabelling, $h_1^+\subset h_2^+$. Since $p\in h_2^-$, we must have $h_2\not\in\H^L(x,y)$, so there must exist a monochromatic $c\in\C$ that separates $x$ from $y$, crosses $h_2$, and has $|c|>L$. But then $c$ crosses both $h_1$ and $h_2$, contradicting the fact that $\{h_1,h_2\}\in\C$.
\end{proof}

The next lemma shows that the gate to $[x,y]$ of any distant wall is uniformly bounded.

\begin{lemma} \label{lem:small_gate} 
Given $x,y \in S_{\C}$, if $h\in\H^L(x,y)$, then $\diam(\g_{[x,y]}(h^-))\leq mL$.
\end{lemma}

\begin{proof} 
Let $a,b\in h^-$, and let $c\in\C$ realise $\dist_\C(\g_{[x,y]}(a),\g_{[x,y]}(b))$. By definition of the gate, $c$ separates $a$ from $b$, and hence every member of $c$ crosses $h$. Since $h$ is distant, every monochromatic subset of $c$ has cardinality at most $L$, so $|c|\le mL$.
\end{proof}

In view of Lemma~\ref{lem:small_gate}, we can consider the distant walls whose gate to $[x,y]$ is close to $x$.

\ubsh{Initial distant walls}
Given $x,y\in S_\C$ and a constant $t$, we write $\H^L_t(x,y)$ for the set of all $h\in\H^L(x,y)$ such that $\g_{[x,y]}(h^-)$ lies in the $t$--ball centred on $x$.
\uesh

We now move to considering the stability properties of the intervals $I(x,y)$. Fix three points $x,y,z\in S_\C$. For ease of notation, we shall write $\mu=\mu(x,y,z)$. Let $r=\dist_\C(x,\mu)$. Our goal is to show that the intersections of $I(x,y)$ and $I(x,z)$ with $B(x,r)$ agree outside a uniform number of balls of uniform size.

\begin{lemma} \label{lem:difference_nonseparating}
If $h\in\H^L_{r-1}(x,y)$, then $h\in\H(x,z)$.
\end{lemma}

\begin{proof}
Since $\{x,y\}\subset h^+$, if $h$ separates $x$ from $z$, then $z\in h^-$. But then $\g_{[x,y]}(z)=\mu$ is at distance $r$ from $x$, contradicting the assumption that $h\in\H^L_{r-1}(x,y)$.
\end{proof}

Lemma~\ref{lem:difference_nonseparating} shows that if $h$ lies in $\H^L_{r-1}(x,y)$, then the only way that $h$ can fail to be in $\H^L(x,z)$ is if $h\in\H(x,z)$ and there is some monochromatic chain in $\C$ of length $L+1$ that crosses $h$ and separates $x$ from $z$.



The following proposition is the key to proving stability of cylinders. The argument is depicted in Figure~\ref{fig:projections} below.

\begin{proposition} \label{prop:few_projections}
Suppose that $h_1,\dots,h_n$ all lie in $\H^L_{r-L-1}(x,y)\ssm\H^L(x,z)$. If we have 
\[
\dist_\C\big(\g_{[x,y]}(h_i^-),\,\g_{[x,y]}(h_j^-)\big) \,>\, L
\]
for all $i\ne j$, then $n\le m$.
\end{proposition}

\begin{proof}
First, according to Lemma~\ref{lem:balls_gated}, the ball $B(x,r-L)$ in $S_\C$ centred on $x$ with radius $r-L$ is gated. Hence there is some $b\in\C$ that separates $\mu$ from $B(x,r-L)$ and has $|b|>L$. In particular, $b$ separates $\mu$ from every $\g_{[x,y]}(h_i^-)$.

By Lemma~\ref{lem:difference_nonseparating}, for each $i$ there is a monochromatic $c_i\in\C$, of colour $\chi_i\in\{1,\dots,m\}$, with $|c_i|>L$, and such that every element of $c_i$ crosses $h_i$ and separates $x$ from $z$.

Suppose for a contradiction that $n>m$. Two of the $c_i$ must have the same colour $\chi$, and after relabelling we can take them to be $c_1$ and $c_2$. Since $\dist_\C(\g_{[x,y]}(h_1^-),\g_{x,y]}(h_2^-))>L$ and every $\g_{[x,y]}(h_i^-)$ is gated (Lemma~\ref{lem:gating_halfspaces}), there exists some chain $a\in\C$ that separates $\g_{[x,y]}(h_1^-)$ from $\g_{[x,y]}(h_2^-)$ with $|a|>L$. Since both gates lie in $[x,y]$, every element of $a$ separates $x$ from $y$, and hence separates $h_1$ from $h_2$. 

Any element of $a$ that does not separate $x$ from $\mu$ must cross every element of $b$, because it both separates $h_1$ from $h_2$ and separates $y$ from $z$. Since $|b|>L$ and $a\in\C$, it follows that at most one element of $a$ can fail to separate $x$ from $\mu$. After relabelling the remaining elements of $a$ separate $h_1$ and $\mu$ from $h_2$ and $x$.

Now consider $c_1$ and $c_2$.  Any element of $c_i$ that does not separate $x$ from $y$ must cross every element of $b$. Since $|b|>L$ and $c_i\in\C$, there can be only one such element of $c_i$. Since $h_i\in\H^L(x,y)$ and $|c_i|>L$, there must actually be one such element, which we denote $h_i'$. In particular, $c_1\ssm\{h_1'\}$ separates $x$ from $\mu$.

Any element of $c_1$ that fails to separate $h_2$ from $\mu$ must cross every element of $a$. Since $|a|>L$ and $c_1\in\C$, at most one element of $c_1$ can fail to separate $h_2$ from $\mu$. Since $|c_1|>L\ge3$, at least two elements $k_1,k_2\in c_1$ separate $h_2$ from $\mu$. But now consider $h_2'$. Since it crosses $h_2$ and separates $z$ from $y$, it crosses both $k_1$ and $k_2$. But $\{k_1,k_2\}\in\D_\chi$ and $h_2'\in W_\chi$, which contradicts the 0--separation of $\D_\chi$ given by Lemma~\ref{lem:disparate_properties}.
\end{proof}

\begin{figure}[ht]
\centering
\includegraphics[height=6.5cm, trim = 0mm 4mm 0mm 4mm]{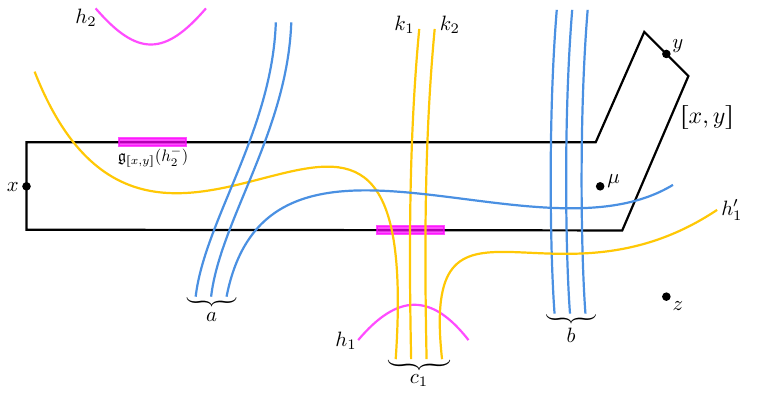}
\caption{Proof of Proposition~\ref{prop:few_projections}. The yellow walls $k_1$ and $k_2$ separate $h_2$ from both $y$ and $z$. Hence no other yellow wall can both cross $h_2$ and separate $y$ from $z$.} \label{fig:projections}
\end{figure}

\begin{corollary} \label{cor:remove_balls} 
There exist balls $B_1,\dots,B_m$ in $S_\C$, each of radius at most $2L(m+1)$, such that for any $h\in\H^L_{r-L-1}(x,y)\ssm\H^L(x,z)$, there is some $i$ such that $\g_{[x,y]}(h^-)\subset B_i$.
\end{corollary}

\begin{proof}
By Proposition~\ref{prop:few_projections}, any collection of such hyperplanes whose gates to $[x,y]$ are pairwise at distance at least $L+1$ has cardinality at most $m$. Let $h_1,\dots,h_n$ be a maximal such collection. Fix $p_i\in\g_{[x,y]}(h_i^-)$, and let $B_i=B(p_i,2mL+L+1)$. By Lemma~\ref{lem:small_gate} and Proposition~\ref{prop:few_projections}, the gate to $[x,y]$ of any other such hyperplane lies in some $B_i$. Since $L\ge3$ we are done.
\end{proof}

We are now in a position to prove the main result of this section. 

\begin{theorem} \label{thm:cylinders}
Let $X$ be a roughly geodesic hyperbolic space, with an action of a group $G$. If $X$ has strong quasitree rank $m<\infty$, then there is some constant $R$ such that the pair $(G,X)$ admits globally $(2m+1,R)$--stable cylinders.
\end{theorem}

\begin{proof}
The result will follow from showing that $S_\C$ has globally stable cylinders. First we define the cylinders. The space $S_\C$ is $3m$--roughly geodesic by Proposition~\ref{prop:dual_SC}. By the Morse lemma, there is a constant $\eps$ such that any two $3m$--rough geodesics in $S_\C$ with the same endpoints are at Hausdorff-distance at most $\eps$. Let $x,y\in S_\C$. Lemma~\ref{lem:interval_quasiline} states that $[x,y]$ contains a $3m$--rough geodesic from $x$ to $y$ and is a uniform quasiline. By Lemma~\ref{lem:cylinder_close}, the same is true for $I(x,y)$. Thus $C(x,y)=\mathcal N(I(x,y),\eps)$ is a cylinder.

It is clear that this choice of cylinders is reversible and $G$--invariant. We show that it is stable. To this end, let $x,y,z\in S_\C$, and write $\mu=\mu(x,y,z)$. Again by Lemma~\ref{lem:interval_quasiline}, there is a uniform constant $\delta$ such that $r=\dist_\C(x,\mu)$ differs from the Gromov product $\rho=\sgen{y,z}_x$ by at most $\delta$.

Suppose that $p\in C(x,z)\ssm C(x,y)$, with $\dist_\C(x,p)\le\rho$. By definition, there exists $q\in I(x,z)\ssm I(x,y)$ with $\dist_\C(q,p)\le\eps$, and we therefore have that $\dist_\C(x,q)\le r+\delta+\eps$. If $\dist_\C(x,q)\ge r-mL-L$, then the fact that $S_\C$ is $3m$--roughly geodesic implies that $q$ lies in a uniform ball centred on $\mu$.

Otherwise, note that $\dist_\C(x,\g_{[x,y]}(q))\le\dist_\C(x,q)$. Since $q\not\in I(x,y)$, there must exist $h\in\H^L(x,y)$ such that $q\in h^-$. Because $q\in I(x,z)$, this must mean that $h\not\in\H^L(x,z)$. The fact that $q\in h^-$ also means that $\g_{[x,y]}(q)\in\g_{[x,y]}(h^-)$. Hence Lemma~\ref{lem:small_gate} implies that $h\in\H^L_{r-L}(x,y)$. Corollary~\ref{cor:remove_balls} now tells us that $\g_{[x,y]}(q)$ lies in one of $m$ balls of radius at most $2L(m+1)$.

We now wish to bound $\dist_\C(q,\g_{[x,y]}(q))$ in the case that $h\in\H^L_{r-L-1}(x,y)$, which will show that $q$, and hence $p$, lies in one of at most $m+1$ balls of uniform size, one of which is centred on $\mu$. Let $c\in\C$ realise $\dist_\C(q,\g_{[x,y]}(q))$. Since $q\in I(x,y)$, Lemma~\ref{lem:cylinder_close} tells us that at most one element of $c$ can separate $q$ from $[x,z]$. The remaining elements of $c$ cross $[x,z]$, and hence must separate $x$ from $z$. They therefore separate $\{q,z\}$ from $\{x,y\}$. We also know that $q\in B(x,r-L)$, so since the latter is gated (Lemma~\ref{lem:balls_gated}), there is some $b\in\C$ that separates $\mu$ from $\{x,q\}$ and has $|b|>L$. But now all but one elements of $c$ cross every element of $b$, so the fact that $c\in\C$ implies that $|c|\le2$.

We have shown that if $p\in C(x,z)\ssm C(x,y)$ has $\dist_\C(x,p)\le\rho$, then $p$ lies in one of at most $m+1$ balls of uniform size, one of which is centred on $\mu$. Reversing the roles of $y$ and $z$, we conclude that $S_\C$ admits globally $(2m+1,R_0)$--stable, $G$--equivariant cylinders, for some uniform constant $R_0$. According to Proposition~\ref{prop:cylinders_qi_invariant} and Remark~\ref{rem:stability}, any roughly geodesic space quasiisometric to $S_\C$ admits globally $(2m+1,R)$--stable cylinders, for some $R$, and if the quasiisometry is $G$--equivariant, then so are the cylinders. The result thus follows from Proposition~\ref{prop:dual_SC}.
\end{proof}

\section{Residually finite hyperbolic groups} \label{sec:rf_hyp} 

Let $G$ be a residually finite hyperbolic group. The following result was proved by Bestvina--Bromberg--Fujiwara.

\begin{theorem}[{\cite[Thm~1.1]{bestvinabrombergfujiwara:proper}}] \label{thm:bbf}
There is a finite product $\prod_{i=1}^mT_i$ of quasitrees on which $G$ admits an action such that its orbit maps are quasiisometric embeddings.
\end{theorem}

In other words, $G$ has property (QT) in the sense of \cite{bestvinabrombergfujiwara:proper}. However, it lacks median information, so it does not state that $G$ has strong (QT) in the sense of Definition~\ref{def:sqtr}. This is provided by the following. Whilst it has not been explicitly stated in the literature, its proof is very similar to those of \cite[Prop.~3.9]{hagenpetyt:projection} and \cite[Prop.~3.2]{petyt:mapping}. We have included an outline of the argument below, and we refer the reader to those papers for more detail. 

\begin{proposition} \label{prop:hyp_orbit_qm}
Orbit maps $G\to\prod_{i=1}^mT_i$ are quasimedian.
\end{proposition}

\begin{proof}[Proof outline]
Let $f$ be an orbit map, and let $H$ be a finite-index subgroup of $G$ that preserves the factors of the product. First note that it suffices to prove that $f|_H$ is quasimedian. Indeed, given $g_1,g_2,g_3\in G$, let $h_i\in H$ be uniformly close to $g_i$. The points $\mu_G(g_1,g_2,g_3)$ and $\mu_G(h_1,h_2,h_3)$ are uniformly close, and, since $f$ is a quasiisometric embedding, the points $\mu_\Pi(fg_1,fg_2,fg_3)$ and $\mu_\Pi(fh_1,fh_2,fh_3)$ are uniformly close. Hence, if $f|_H$ is quasimedian, then $f$ is quasimedian.

Let $f_i$ be the orbit of $H$ on $T_i$ induced by $f|_H$. Since the median on $\prod_{i=1}^mT_i$ is defined component-wise, it suffices to show that each $f_i$ is quasimedian. For that, it is enough to show that geodesics in $H$ get mapped to unparametrised quasigeodesics in $T_i$. This amounts to showing that if $\gamma\subset H$ is a geodesic, then $f_i\gamma$ does not backtrack too much.

The quasitrees $T_i$ are constructed from a collection $\cal A$ of quasigeodesics in $G$. For each $\alpha\in\cal A$, we can consider the closest-point projection $\pi_\alpha:\gamma\to\alpha$. In the case where $\pi_\alpha\gamma$ is a large subinterval, one can show that if $\pi_\alpha(x)$ is far from the endpoints of that interval, then $f_i(x)$ can coarsely be seen as the inclusion of $\pi_\alpha(x)$ in $T_i$. One can also show that, for each subpath $\gamma'$ of $\gamma$ not projecting to the interior of any such long subinterval of any $\alpha$, the image $f_i\gamma'$ is bounded.

By a careful choice of which subintervals to consider, one can then realise $f_i\gamma$ as a concatenation of unparametrised quasigeodesics that do not fellow-travel one another, which is an unparametrised quasigeodesic itself.  
\end{proof}

We can now apply the results of the previous sections to prove the following.

\begin{theorem} \label{thm:cylinders:hyperbolic}
Every residually finite hyperbolic group admits globally stable cylinders.
\end{theorem}

\begin{proof}
Let $G$ be a residually finite hyperbolic group. By Theorem~\ref{thm:bbf} and Proposition~\ref{prop:hyp_orbit_qm}, $G$ has strong (QT). The result follows from Theorem~\ref{thm:cylinders}.
\end{proof}

\section{Curve graphs} \label{sec:curve_graph} 

Let $\Sigma$ be a finite-type surface. The goal of this section is to show that the curve graph $\C\Sigma$ admits globally stable $\MCG\Sigma$--equivariant cylinders. In view of Theorem~\ref{thm:cylinders}, the main step remaining is to  show that $(\C\Sigma,\MCG\Sigma)$ has strong (QT). This will ultimately be a consequence of the following, which is a combination of \cite[Thm~1.2]{bestvinabrombergfujiwara:proper} and \cite[Prop.~3.2]{petyt:mapping}.

\begin{theorem} \label{thm:mcg_rank}
Mapping class groups have strong (QT): they admit actions on finite products of quasitrees so that orbit maps are quasimedian quasiisometric embeddings.
\end{theorem}

The quasitrees give dualisable systems on the product, as in Section~\ref{sec:QT}, and our strategy will be to refine these dualisable systems to obtain a new collection of quasitrees so that the orbit of the mapping class group is quasiisometric to the curve graph.

Let $\prod_{i=1}^mT_i$ be given by Theorem~\ref{thm:mcg_rank}, and let $f:\MCG\Sigma\to\prod_{i=1}^mT_i$ be an orbit map. Let $W_i$ and $\D_i$ be as in Section~\ref{sec:QT}, let $S$ denote the orbit of $\MCG\Sigma$ on $\prod_{i=1}^mT_i$. Let $W$, $\D^1$, and $\D$ be as in Section~\ref{sec:dualising_subspaces}. Specifically, $W$ consists of all walls induced by elements of the $W_i$, but with crossing data from $S$. Note that $\dist_{\D^1}$ is exactly the $\ell^1$ metric on $S$.

\begin{lemma} \label{lem:d1_and_d}
The identity map $(S,\dist_{\D^1})\to(S,\dist_\D)$ is a quasiisometry.
\end{lemma}

\begin{proof}
Since $\D\subset\D^1$, we have $\dist_\D\le\dist_{\D^1}$. For the reverse, if $d=(d_i)_{i=1}^m\in\D^1$, then there is some $i$ such that $|d_i|\ge\frac{|d|}m$. Hence $\dist_\D\ge\frac1m\dist_{\D^1}$.
\end{proof}

By pulling back along $f$, each $h\in W$ induces a wall $\bar h$ of $\MCG\Sigma$. Since $f$ is quasimedian and the halfspaces of $h$ are median-quasiconvex in $S$, the halfspaces of $\bar h$ are median-quasiconvex. If $\overline W$ denotes the set of such walls on $\MCG\Sigma$, then every dualisable system on $(S,W)$ has an associated dualisable system on $(\MCG\Sigma,\overline W)$, and the associated metrics on $S$ and $\MCG\Sigma$ are quasiisometric via $f$.

For a constant $K$, define $\E^K=\{d\in\D\,:\,d\text{ is a }K\text{--chain}\}$. The following proposition will be the most technical aspect of our argument. See Section~\ref{subsec:mcg} for terminology.

\begin{proposition} \label{prop:curve_graph}
For sufficiently large $K$, the space $(S,\dist_{\E^K})$ is quasiisometric to the curve graph $\C\Sigma$.
\end{proposition}

\begin{proof}
In view of the above remarks, we shall abuse notation and view $\E^K$ as a dualisable system on $\MCG\Sigma$, where the walls in $W$ all have median-quasiconvex halfspaces. Recall that $\MCG\Sigma$ has an equivariant, coarsely Lipschitz, quasimedian projection map $\pi:\MCG\Sigma\to\C\Sigma$. We shall prove that $\pi:(\MCG\Sigma,\dist_{\E^K})\to(\C\Sigma,\dist_\Sigma)$ is a quasiisometry for large $K$. Let $E$ be a constant as in Remark~\ref{rem:E}, and let $R$ be sufficiently large compared to $E$. 

Let $x,y\in\MCG\Sigma$. First we show that $\dist_\Sigma(\pi(x),\pi(y))$ is a coarse lower-bound for $\dist_{\E^K}(x,y)$. Let $\gamma$ be a geodesic in $\C\Sigma $ from $\pi(x)$ to $\pi(y)$. Let $\gamma_1,\dots,\gamma_n\subset\gamma$ be a maximal sequence of unit subintervals of $\gamma$ with $\dist_\Sigma(\gamma_i,\gamma_{i+1})>R$ for all $i$. Let $\pi_\gamma:\C\Sigma\to\gamma$ be the closest-point projection map, and for each $i$ set $J_i=\pi_\gamma^{-1}\gamma_i$.

Now, for each $i$ let $H_i=\hull(\pi^{-1}J_i)\subset\MCG\Sigma$. Since hulls are obtained by taking medians at most $\xi$ times, where $\xi$ is the complexity of $\Sigma$, and the map $\pi$ is quasimedian, we have that $\pi(H_i)$ lies in a uniform neighbourhood of $\gamma_i$, with constant depending only on $E$ and $\xi$. In particular, the sets $\pi(H_i)$ are pairwise at distance almost $R$. Since $\pi$ is $E$--coarsely Lipschitz with respect to the usual metric on $\MCG\Sigma$, it follows from Lemma~\ref{lem:d1_and_d} that $\dist_\D(H_i,H_{i+1})$ is bounded below by a uniformly linear function of $R$. Thus, because halfspaces are median-quasiconvex, if $R$ is sufficiently large then there exists $c_i=(h_{i,-m},h_{i,1-m},\dots,h_{i,0},\dots,h_{i,m})\in\D$ separating $H_i$ from $H_{i+1}$.

As another consequence of the fact that the $\pi(H_i)$ are pairwise distant, \cite[Lem.~7.10]{durhamzalloum:geometry} states that the coarse gate $\p_{H_{i+1}}(H_i)$ has $\MCG\Sigma$--diameter bounded uniformly in terms of $E$, and vice versa. If $p,q\in\p_{H_{i+1}}(H_i)$ and $c\in\D$ separates $p$ from $q$, then median-quasiconvexity of halfspaces implies that almost all elements of $c$ cross $H_i$. In particular, Lemma~\ref{lem:d1_and_d} gives a uniform bound $K$ on the cardinality of any such $c$.

If a wall $h\in W$ crosses an element of $c_i$ and an element of $c_{i+2}$, then it crosses both $H_{i+1}$ and $H_{i+2}$. Thus $h_{i,0}\in c_i$ and $h_{i+2,0}\in c_{i+2}$ are $K$--separated. Since $\D$ is $m$--gluable (Lemma~\ref{lem:C_is_gluable}), the chain $(h_{1,0},h_{3,0},h_{5,0},\dots)$ lies in $\D$, and hence in $\E^K$. This gives the desired lower bound.

For the reverse, we shall use the fact that $\C\Sigma$ can, up to quasiisometry, be obtained from $\MCG\Sigma$ by coning-off all curve stabilisers, each of which is coarsely a product $\mathbf{P}_\alpha$ \cite[Thm~1.3]{masurminsky:geometry:1}. In view of this, it suffices to show that such product regions are bounded in the metric $\dist_{\E^K}$, for that will imply that $\dist_\Sigma$ is a coarse upper-bound for $\dist_{\E^K}$.

Given any $p,q\in\mathbf P_\alpha$, we can find a square in $\mathbf P_\alpha$, two of whose vertices are $p$ and $q$. Let us cyclically label the vertices of the square $p_1,p_2,p_3,p_4$, so that the median $\mu(p_{i-1},p_i,p_{i+1})$ is uniformly close to $p_i$ for each $i$. Since halfspaces are median-quasiconvex, if $c_1\in\D$ separates $p_1$ from $p_4$, then almost all elements of $c_1$ separate $p_2$ from $p_3$. Similarly, if $c_2\in\D$ separates $p_2$ from $p_1$, then almost all elements of $c_2$ separate $p_4$ from $p_3$. By Lemma~\ref{lem:d1_and_d}, this shows that any long element of $\D$ separating $p$ from $q$ forms part of a large grid. This bounds the diameter of $\mathbf P_\alpha$ in $\dist_{\E^K}$, completing the proof.
\end{proof}

We now fix a value of $K$ as in Proposition~\ref{prop:curve_graph}, and let $\E=\E^K$. For each $i\in\{1,\dots,m\}$, let $\E_i$ be the elements of $\E$ induced by $\D_i$. Unlike the situation where $S$ is hyperbolic, $T_i$ will not be quasiisometric to $(T_i,\dist_{\E_i})$. The next two lemmas show that the latter is nevertheless a quasitree.

\begin{lemma} \label{lem:Ei_quasitree}
$\E_i$ is a 1--gluable and 0--separated dualisable system of chains on $(T_i,W_i)$. The dual space $T_{\E_i}$ is a quasitree.
\end{lemma}

\begin{proof}
Since $\E_i\subset\D_i$, it is 0--separated by Lemma~\ref{lem:disparate_properties}. If $c_1,c_2\in\E_i$ are such that $c_1\cup c_2$ is a chain, then let $c\subset c_1\cup c_2$ be obtained by removing the last element of $c_1$. Since $\D_i$ is 1--gluable, $c\in\D_i$. Moreover, any element of $\D$ that crosses two elements of $c$ must cross two elements of one of the $c_i$, and so has length at most $K$. Hence $c\in\E_i$. The final statement is given by Proposition~\ref{prop:0-sep_quasitree}.
\end{proof}

\begin{lemma} \label{lem:Ei_dense}
Geodesics in $T_i$ are unparametrised rough geodesics in $T_{\E_i}$. Hence $(T_i,\dist_{\E_i})$ is coarsely dense in the quasitree $T_{\E_i}$, and the natural map $(T_i,\dist_{\E_i})\to T_{\E_i}$ is quasimedian.
\end{lemma}

\begin{proof}
Let $s_1,s_2\in T_i$, and let $\gamma$ be a geodesic (for the original metric) connecting them. Let $t_1,p,t_2\in\gamma$, in that order. Let $c_1\in\E_i$ realise $\dist_{\E_i}(t_1,p)$, and let $c_2\in\E_i$ realise $\dist_{\E_i}(p,t_2)$. Since elements of $\D_i$ arise from disparate sets of balls, if $c'_1$ is obtained from $c_1$ by removing the element closest to $p$, then $c'_1\cup c_2$ is a chain. By 1--gluability, $\dist_{\E_i}(t_1,p)+\dist_{\E_i}(p,t_2)\le2+\dist_{\E_i}(t_1,t_2)$. This shows that $\gamma$ in an unparametrised rough geodesic in $T_{\E_i}$. This implies that the map is quasimedian, and the coarse density follows from Proposition~\ref{prop:density}.
\end{proof}

We now consider the product $\prod_{i=1}^mT_{\E_i}$ with the $\ell^1$ metric. Note that if we set
\[
\E' \,=\, \big\{\bigcup_{i=1}^md_i\,:\,d_i\in\E_i\big\},
\]
then the metric $\dist_{\E'}$ is exactly the $\ell^1$ metric on $\prod_{i=1}^mT_{\E_i}$. We are ready to prove the following.

\begin{theorem} \label{thm:curve_graph_rank}
The pair $(\C\Sigma,\MCG\Sigma)$ has strong (QT).
\end{theorem}

\begin{proof}
Proposition~\ref{prop:curve_graph} tells us that $\C\Sigma$ is quasiisometric to $(S,\dist_\E)$. The maps $\pi$ and $f$ are both quasimedian, and Lemma~\ref{lem:Ei_dense} implies that the natural map $\prod_{i=1}^mT_i\to\prod_{i=1}^mT_{\E_i}$ is quasimedian as well. It therefore suffices to show that the natural map $(S,\dist_\E)\to(S,\dist_{\E'})\subset \prod_{i=1}^mT_{\E_i}$ is a quasiisometry.

Since $\E\subset\E'$, we have that $\dist_\E\le\dist_{\E'}$. For the reverse, given any $c=(d_i)_{i=1}^m\in\E'$, there is some $i$ such that $|d_i|\ge\frac{|c|}m$. As $d_i\in\E$, this means that $\dist_\E\ge\frac1m\dist_{\E'}$.
\end{proof}

By the results of Section~\ref{sec:stable}, we obtain the following.

\begin{corollary} \label{cor:curve_graph_cylinders}
The pair $(\C\Sigma, \MCG\Sigma)$ admits globally stable cylinders.
\end{corollary}

\begin{proof}
Theorem~\ref{thm:curve_graph_rank} states that $(\C\Sigma,\MCG\Sigma)$ has strong (QT), so the result is given by Theorem~\ref{thm:cylinders}.
\end{proof}

\begin{remark} \label{rem:generality}
Although we have worked with mapping class groups and curve graphs in this section, the arguments actually apply for any hierarchically hyperbolic group with strong (QT), together with its \emph{largest} hyperbolic space, as constructed in \cite{abbottbehrstockdurham:largest}. Indeed, the only point that needs commenting on is the part of the proof of Proposition~\ref{prop:curve_graph} where we applied \cite[Thm~1.3]{masurminsky:geometry:1}: the analogous statement holds for all HHGs by \cite[Cor.~H]{behrstockhagensisto:asymptotic}.
\end{remark}

\begin{theorem} \label{thm:hhg}
Let $G$ be a hierarchically hyperbolic group, with largest hyperbolic space $X$. If $G$ has strong (QT) then $X$ has strong (QT), and $(G,X)$ admits globally stable cylinders.
\end{theorem}

The class of hierarchically hyperbolic groups with strong (QT) is much larger than just mapping class groups and residually finite hyperbolic groups. The most general result of this form is given by recent work of Tao \cite{tao:property}, generalising results from \cite{bestvinabrombergfujiwara:proper,hagenpetyt:projection,hannguyenyang:property}. One of the many classes to which his work applies is that of Artin groups of \emph{large} and \emph{hyperbolic} type \cite[Thm~1.4]{tao:property}. For those groups, it was noted in \cite[Cor.~B, Lem.~6.18]{hagenmartinsisto:extra} that the largest hyperbolic space is the \emph{coned-off Deligne complex} $\hat D$ defined in \cite{martinprzytycki:acylindrical}. The above arguments therefore prove the following.

\begin{corollary} \label{cor:artin}
Let $A$ be a residually finite Artin group of large and hyperbolic type. The coned-off Deligne complex $\hat D$ has strong (QT), and $(A,\hat D)$ admits globally stable cylinders.
\end{corollary}

The question of which Artin groups are residually finite remains an active topic of research, but it is known for many large-type Artin groups thanks to work of Jankiewicz \cite{jankiewicz:residual}.

\bibliographystyle{alpha}
\bibliography{bibtex}
\end{document}